\documentclass[11pt]{amsart}%
\usepackage{amsfonts}
\usepackage{amsmath}
\usepackage{amssymb}
\usepackage{graphicx}%
\providecommand{\U}[1]{\protect\rule{.1in}{.1in}}
\newtheorem{theorem}{Theorem}
\theoremstyle{plain}

\newtheorem{corollary}{Corollary}

\newtheorem{lemma}{Lemma}

\newtheorem{problem}{Problem}
\newtheorem{proposition}{Proposition}

\numberwithin{equation}{section}
\begin{document}
\title[On the Bohnenblust--Hille inequality]{On the growth of the optimal constants of the multilinear Bohnenblust--Hille inequality}
\author{Daniel Nu\~{n}ez-Alarc\'{o}n and Daniel Pellegrino }
\address{Departamento de Matem\'{a}tica, UFPB, Jo\~{a}o Pessoa, PB, Brazil}
\email{dmpellegrino@gmail.com}

\begin{abstract}
Let $\left(  K_{n}\right)  _{n=1}^{\infty}$ be the optimal constants
satisfying the multilinear (real or complex) Bohnenblust--Hille inequality.
The exact values of the constants $K_{n}$ are still waiting to be discovered
since eighty years ago; recently, it was proved that $\left(  K_{n}\right)
_{n=1}^{\infty}$ has a subexponential growth. In this note we go a step
further and address the following question: Is it true that
\[
\lim_{n\rightarrow\infty}\left(  K_{n}-K_{n-1}\right)  =0?
\]
Our main result is a Dichotomy Theorem for the constants satisfying the
Bohnenblust--Hille inequality; in particular we show that the answer to the
above problem is essentially positive in a sense that will be clear along the
note. Another consequence of the dichotomy proved in this note is that
$\left(  K_{n}\right)  _{n=1}^{\infty}$ has a kind of subpolynomial growth: if
$p(n)$ is any non-constant polynomial, then $K_{n}$ \textit{is} \textit{not}
asymptotically equal to $p(n).$ Moreover, if
\[
q>\log_{2}\left(  \frac{e^{1-\frac{1}{2}\gamma}}{\sqrt{2}}\right)
\approx0.526,
\]
then
\[
K_{n}\nsim n^{q}.
\]

\end{abstract}
\maketitle

\section{Introduction}

Let $\mathbb{K}$ be the real or complex scalar field. The multilinear
Bohnenblust--Hille inequality (see, for example, \cite{bh, defant, defant2,
Mu} and also \cite{annals} for a polynomial version) asserts that for every
positive integer $n\geq1$ there exists a constant $c_{\mathbb{K},n}$ such
that
\begin{equation}
\left(  \sum\limits_{i_{1},\ldots,i_{n}=1}^{N}\left\vert T(e_{i_{^{1}}}%
,\ldots,e_{i_{n}})\right\vert ^{\frac{2n}{n+1}}\right)  ^{\frac{n+1}{2n}}\leq
c_{\mathbb{K},n}\sup_{z_{1},...,z_{n}\in\mathbb{D}^{N}}\left\vert
T(z_{1},...,z_{n})\right\vert \label{hypp}%
\end{equation}
for all $n$-linear form $T:\mathbb{K}^{N}\times\cdots\times\mathbb{K}%
^{N}\rightarrow\mathbb{K}$ and every positive integer $N$, where $\left(
e_{i}\right)  _{i=1}^{N}$ denotes the canonical basis of $\mathbb{K}^{N}$ and
$\mathbb{D}^{N}$ represents the open unit polydisk in $\mathbb{K}^{N}$. It is
well-known that $c_{\mathbb{K},n}\in\lbrack1,\infty)$ for all $n$ and that the
power $\frac{2n}{n+1}$ is sharp but, on the other hand, the optimal values for
$c_{\mathbb{K},n}$ remain a mystery. To the best of our knowledge the unique
precise information is that $c_{\mathbb{R},2}=\sqrt{2}$ is sharp (see
\cite{diniz2}). The original constants obtained by Bohnenblust and Hille (for
the complex case) are
\[
c_{\mathbb{C},n}=n^{\frac{n+1}{2n}}2^{\frac{n-1}{2}}.
\]
Later, these results were improved to%
\[
c_{\mathbb{C},n}=2^{\frac{n-1}{2}}\text{ (Davie, Kaisjer, 1973 (\cite{d,
Ka}))},
\]
and%
\[
c_{\mathbb{C},n}=\left(  \frac{2}{\sqrt{\pi}}\right)  ^{n-1}\text{
(Qu\'{e}ffelec, 1995 (\cite{Q})).}%
\]
In 2012 (\cite{jfa}) it was proved that the best constants satisfying the
Bohnenblust--Hille inequality have a subexponential growth (for both real and
complex scalars). A step further would be to verify if these optimal constants
have an even better asymptotic behavior. More precisely:

\begin{problem}
Let $\left(  K_{\mathbb{K},n}\right)  _{n=1}^{\infty}$ be the sequence of
optimal constants satisfying the Bohnenblust--Hille inequality. Is it true
that%
\begin{equation}
\lim_{n\rightarrow\infty}\left(  K_{\mathbb{K},n}-K_{\mathbb{K},n-1}\right)
=0? \label{bh}%
\end{equation}

\end{problem}

In this note, among other results, we essentially show that the answer to this
problem is positive. More precisely, as a consequence of our main result
(Dichotomy Theorem) we show that if there exist $L_{1},L_{2}\in$ $\left[
0,\infty\right]  $ so that
\[
L_{1}=\lim_{n\rightarrow\infty}\left(  K_{\mathbb{K},n}-K_{\mathbb{K}%
,n-1}\right)  \text{ and }L_{2}=\lim_{n\rightarrow\infty}\frac{K_{\mathbb{K}%
,2n}}{K_{\mathbb{K},n}}%
\]
then
\[
L_{1}=0
\]
and
\[
L_{2}\in\lbrack1,\frac{e^{1-\frac{1}{2}\gamma}}{\sqrt{2}}],
\]
where $\gamma$ denotes the Euler constant (see (\ref{h77})). The non-existence
of the above limits would be an extremely odd event since there is no reason
for a pathological behavior for the optimal constants $\left(  K_{\mathbb{K}%
,n}\right)  _{n=1}^{\infty}$ satisfying the Bohnenblust--Hille inequality$.$

Another corollary of the Dichotomy Theorem is that the sequence $\left(
K_{\mathbb{K},n}\right)  _{n=1}^{\infty}$ of optimal constants satisfying the
Bohnenblust--Hille inequality \textit{can not} have any kind of polynomial
growth. Also, if
\[
q>\log_{2}\left(  \frac{e^{1-\frac{1}{2}\gamma}}{\sqrt{2}}\right)
\approx0.526,
\]
then%
\[
K_{\mathbb{K},n}\nsim n^{q}.
\]

\section{The Dichotomy Theorem}

From now on our arguments will hold for both real and complex scalars, so we
will use the same notation for both cases. In all this note $\left(
K_{n}\right)  _{n=1}^{\infty}$ denotes the sequence of the optimal constants
satisfying the Bohnenblust--Hille inequality.

From now on we say that a sequence of positive real numbers $\left(
R_{n}\right)  _{n=1}^{\infty}$ is \textit{well-behaved} if there are
$L_{1},L_{2}\in\lbrack0,\infty]$ such that%
\begin{equation}
\lim_{n\rightarrow\infty}\frac{R_{2n}}{R_{n}}=L_{1} \label{um}%
\end{equation}
and%
\begin{equation}
\lim_{n\rightarrow\infty}\left(  R_{n}-R_{n-1}\right)  =L_{2}. \label{dois}%
\end{equation}

Note that the above requirements are quite weak (observe that $L_{1},L_{2}$
may be infinity). So, any sequence of the form%
\begin{align*}
R_{n}  &  =ba^{cn}\text{ for }\left(  a,b,c\right)  \in\lbrack0,\infty
)^{2}\times(-\infty,\infty),\text{ or}\\
R_{n}  &  =ba^{\frac{c}{n}}\text{ for }\left(  a,b,c\right)  \in
\lbrack0,\infty)^{2}\times(-\infty,\infty),\text{ or}\\
R_{n}  &  =bn^{a}\text{ for }\left(  a,b\right)  \in(-\infty,\infty
)\times\lbrack0,\infty)\text{, or}\\
R_{n}  &  =b\log n,\text{ for }b\in(0,\infty)\text{ or}\\
R_{n}  &  =%
{\textstyle\sum\limits_{j=0}^{k}}
a_{j}n^{j}\text{ with }a_{k}>0
\end{align*}
is well-behaved. Since the elements of $\left(  K_{n}\right)  _{n=1}^{\infty}$
belong to $[1,\infty)$, we will restrict our attention to well-behaved
sequences in $[1,\infty).$ We also remark that, even restricted to sequences
in $[1,\infty),$ the limits (\ref{um}) and (\ref{dois}) are, in fact,
independent. For example%
\begin{equation}
R_{n}:=\left\{
\begin{array}
[c]{c}%
\sqrt{n}\text{, if }n=2^{k}\text{ for some }k,\\
2\sqrt{n},\text{ otherwise}%
\end{array}
\right.  \label{contra}%
\end{equation}
satisfies (\ref{um}) with $L_{1}=\sqrt{2}$ but does not fulfil (\ref{dois}).
On the other hand let, for all positive integers $k>1,$%
\[
B_{k}:=\{2^{k}-1,...,2^{k+1}-2\}.
\]
The sequence
\[
R_{n}:=\left\{
\begin{array}
[c]{c}%
n^{k}\text{, if }n\in B_{k}\text{ for some }k\text{ odd},\\
\left(  \min B_{k}\right)  ^{k}+kn\text{, if }n\in B_{k}\text{ for some
}k\text{ even},
\end{array}
\right.
\]
satisfies (\ref{dois}) but does not satisfy (\ref{um}).

Henceforth the subexponential sequence of constants satisfying the multilinear
Bohnenblust--Hille inequality constructed in \cite{jfa} is denoted by $\left(
C_{n}\right)  _{n=1}^{\infty}.$ Since we are interested in the growth of
$\left(  K_{n}\right)  _{n=1}^{\infty}$, we will restrict our attention to
sequences $\left(  R_{n}\right)  _{n=1}^{\infty}$ so that $1\leq R_{n}\leq
C_{n}$ for all $n$.

Our main result is the following dichotomy:

\bigskip

\textbf{Dichotomy Theorem.} If $1\leq R_{n}\leq C_{n}$ for all $n$, then
exactly one of the following assertions is true:

(i) $\left(  R_{n}\right)  _{n=1}^{\infty}$ is subexponential and not well-behaved.

(ii) $\left(  R_{n}\right)  _{n=1}^{\infty}$ is well-behaved with
\begin{equation}
\lim_{n\rightarrow\infty}\frac{R_{2n}}{R_{n}}\in\lbrack1,\frac{e^{1-\frac
{1}{2}\gamma}}{\sqrt{2}}] \label{i1}%
\end{equation}
and
\begin{equation}
\lim_{n\rightarrow\infty}\left(  R_{n}-R_{n-1}\right)  =0. \label{i2}%
\end{equation}

As a corollary we extract the following information on the optimal constants
$\left(  K_{n}\right)  _{n=1}^{\infty}:$

\begin{corollary}
The optimal constants $\left(  K_{n}\right)  _{n=1}^{\infty}$ satisfying the
Bohnenblust--Hille inequality is

(i) subexponential and not well-behaved

or

(ii) well-behaved with
\[
\lim_{n\rightarrow\infty}\frac{K_{2n}}{K_{n}}\in\lbrack1,\frac{e^{1-\frac
{1}{2}\gamma}}{\sqrt{2}}]
\]
and
\[
\lim_{n\rightarrow\infty}\left(  K_{n}-K_{n-1}\right)  =0.
\]

\end{corollary}

If (i) is true, then we will have a completely surprising result: the bad
behavior of $\left(  K_{n}\right)  _{n=1}^{\infty}$. On the other hand, if
(ii) is true we will have an almost ultimate and surprising information on the
growth of $\left(  K_{n}\right)  _{n=1}^{\infty}$.

We remark that there exist well-behaved sequences $\left(  R_{n}\right)
_{n=1}^{\infty}$ such that
\[
\lim_{n\rightarrow\infty}\left(  R_{n}-R_{n-1}\right)  =0
\]
but
\[
\lim_{n\rightarrow\infty}\frac{R_{2n}}{R_{n}}\notin\lbrack1,\frac
{e^{1-\frac{1}{2}\gamma}}{\sqrt{2}}].
\]
In fact, since $\frac{e^{1-\frac{1}{2}\gamma}}{\sqrt{2}}\approx1.44$, for
\[
R_{n}=n^{\frac{3}{5}}%
\]
we have
\[
\lim_{n\rightarrow\infty}\left(  R_{n}-R_{n-1}\right)  =0
\]
and
\[
\lim_{n\rightarrow\infty}\frac{R_{2n}}{R_{n}}=2^{\frac{3}{5}}\notin
\lbrack1,\frac{e^{1-\frac{1}{2}\gamma}}{\sqrt{2}}].
\]
On the other hand, as a consequence of our results we observe the simple (but
useful) fact: if $\left(  R_{n}\right)  _{n=1}^{\infty}$ is well-behaved and%
\[
\lim_{n\rightarrow\infty}\frac{R_{2n}}{R_{n}}\in\lbrack1,2),
\]
then necessarily
\[
\lim_{n\rightarrow\infty}\left(  R_{n}-R_{n-1}\right)  =0.
\]
So, \textit{a fortiori}, condition (\ref{i2}) is superfluous.

\section{The proofs}

Henceforth the letter $\gamma$ denotes the Euler constant
\begin{equation}
\gamma=\lim_{m\rightarrow\infty}\left(  \left(  -\log m\right)  +\sum
_{k=1}^{m}\frac{1}{k}\right)  \approx0.577.\label{h77}%
\end{equation}

\begin{lemma}
\label{ll1}If $1\leq R_{n}\leq C_{n}$ for all $n$ and there exists
$L\in\lbrack0,\infty]$ so that
\begin{equation}
L=\lim_{n\rightarrow\infty}\frac{R_{2n}}{R_{n}},\text{ } \label{zz}%
\end{equation}
then%
\[
L\in\lbrack1,\frac{e^{1-\frac{1}{2}\gamma}}{\sqrt{2}}].
\]

\end{lemma}

\begin{proof}
Suppose that $L<1.$ For any $0<\varepsilon<1,$ there is a $N_{0}$ so that%
\[
n\geq N_{0}\Rightarrow\frac{R_{2n}}{R_{n}}<1-\varepsilon.
\]
Arguing by induction we have%
\[
R_{2^{l}N_{0}}<R_{N_{0}}(1-\varepsilon)^{l}%
\]
for all positive integer $l$ and we conclude that%
\[
\lim_{l\rightarrow\infty}R_{2^{l}N_{0}}=0,
\]
which is impossible, since $R_{n}\geq1$ for all $n.$ To simplify the notation,
we will write%
\[
\alpha:=\frac{e^{1-\frac{1}{2}\gamma}}{\sqrt{2}}.
\]
Now let us show that $L>\alpha$ is also not possible. From \cite{jfa} we know
that%
\[
\lim_{n\rightarrow\infty}\frac{C_{2n}}{C_{n}}=\alpha.
\]
We will show that there is a sufficiently large $N$ so that $R_{N}>C_{N}$
(which is a contradiction).

Given a small $0<\varepsilon<L-\alpha$, there is a $n_{0}$ so that%
\[
n\geq n_{0}\Rightarrow\frac{C_{2n}}{C_{n}}<\alpha+\frac{\varepsilon}%
{2}:=A\text{ and }\frac{R_{2n}}{R_{n}}>\alpha+\varepsilon:=B.
\]
Using induction we have%
\begin{align*}
C_{2^{l}n_{0}}  &  <A^{l}C_{n_{0}}\\
R_{2^{l}n_{0}}  &  >B^{l}R_{n_{0}}%
\end{align*}
for all positive integer $l$. Hence%
\[
\frac{R_{2^{l}n_{0}}}{C_{2^{l}n_{0}}}>\frac{B^{l}R_{n_{0}}}{A^{l}C_{n_{0}}%
}=\left(  \frac{B}{A}\right)  ^{l}\frac{R_{n_{0}}}{C_{n_{0}}}.
\]
Since $\frac{B}{A}>1$ we conclude that%
\[
\lim_{l\rightarrow\infty}\left(  \frac{B}{A}\right)  ^{l}\frac{R_{n_{0}}%
}{C_{n_{0}}}=\infty
\]
and thus there is a positive integer $N_{1}$ so that%
\[
\frac{R_{2^{N_{1}}n_{0}}}{C_{2^{N_{1}}n_{0}}}>1,
\]
which is a contradiction.
\end{proof}

\begin{lemma}
If $1\leq R_{n}\leq C_{n}$ for all $n$, and the limit (\ref{zz}) exists, then
there is a positive integer $N_{0}$ so that%
\[
\frac{2^{l}N_{0}}{R_{2^{l}N_{0}}}>\frac{N_{0}}{R_{N_{0}}}\left(  \frac{4}%
{3}\right)  ^{l}%
\]
for all positive integers $l.$
\end{lemma}

\begin{proof}
From the previous lemma we have
\[
\lim_{n\rightarrow\infty}\frac{R_{2n}}{R_{n}}=L\in\lbrack1,\frac{e^{1-\frac
{1}{2}\gamma}}{\sqrt{2}}].
\]
So, since%
\[
\frac{e^{1-\frac{1}{2}\gamma}}{\sqrt{2}}<\frac{3}{2},
\]
let us fix $N_{0}$ so that%
\[
\frac{R_{2n}}{R_{n}}<\frac{3}{2}%
\]
for all $n\geq N_{0}.$ Hence, by induction,%
\[
R_{2^{l}N_{0}}<\left(  \frac{3}{2}\right)  ^{l}R_{N_{0}}%
\]
for all positive integer $l$. We conclude that%
\[
\frac{2^{l}N_{0}}{R_{2^{l}N_{0}}}>\frac{2^{l}N_{0}}{\left(  \frac{3}%
{2}\right)  ^{l}R_{N_{0}}}=\frac{N_{0}}{R_{N_{0}}}\left(  \frac{4}{3}\right)
^{l}%
\]
for all $l.$
\end{proof}

\begin{lemma}
If $\left(  R_{n}\right)  _{n=1}^{\infty}$ is well-behaved and $1\leq
R_{n}\leq C_{n}$ for all $n,$ then%
\begin{equation}
\lim_{n\rightarrow\infty}\left(  R_{n}-R_{n-1}\right)  =0 \label{est}%
\end{equation}

\end{lemma}

\begin{proof}
Let $M:=\lim_{n\rightarrow\infty}\left(  R_{n}-R_{n-1}\right)  .$ The first
(and main) step is to show that $M\notin(0,\infty).$

Suppose that $M\in(0,\infty)$. In this case, from (\ref{est}) there is a
positive integer $N_{1}$ such that%
\[
n\geq N_{1}\Rightarrow R_{n}-R_{n-1}>\frac{M}{2}.
\]
So,
\[
n\geq N_{1}\Rightarrow R_{2n}-R_{n}>\frac{nM}{2}%
\]
Hence%
\[
n\geq N_{1}\Rightarrow\frac{R_{2n}}{R_{n}}-1>\left(  \frac{nM}{2}\right)
\frac{1}{R_{n}}.
\]
From Lemma \ref{ll1} we have
\[
1\leq\lim_{n\rightarrow\infty}\frac{R_{2n}}{R_{n}}\leq\frac{e^{1-\frac{1}%
{2}\gamma}}{\sqrt{2}}.
\]
So, there is a positive integer $N_{2}$ so that%
\[
n\geq N_{2}\Rightarrow\frac{R_{2n}}{R_{n}}<\frac{3}{2}.
\]
Hence if $N_{3}=\max\left\{  N_{1},N_{2}\right\}  $ we have
\begin{equation}
n\geq N_{3}\Rightarrow\frac{1}{2}>\frac{R_{2n}}{R_{n}}-1>\left(  \frac{nM}%
{2}\right)  \frac{1}{R_{n}}. \label{hyz}%
\end{equation}
From the previous lemma we know that there is a $N_{0}$ so that%
\begin{equation}
\frac{2^{l}N_{0}}{R_{2^{l}N_{0}}}>\frac{N_{0}}{R_{N_{0}}}\left(  \frac{4}%
{3}\right)  ^{l} \label{ab1}%
\end{equation}
for all positive integers $l.$

Now we choose a positive integer $l_{0}$ such that%
\[
l>l_{0}\Rightarrow N_{l}:=2^{l}N_{0}>N_{3}.
\]
Hence, from (\ref{hyz}), we know that
\begin{equation}
l\geq l_{0}\Rightarrow N_{l}>N_{3}\Rightarrow\frac{1}{2}>\left(  \frac{N_{l}%
M}{2}\right)  .\frac{1}{R_{N_{l}}}=\left(  \frac{M}{2}\right)  .\frac
{2^{l}N_{0}}{R_{2^{l}N_{0}}}. \label{ab2}%
\end{equation}
But, since
\[
\lim_{l\rightarrow\infty}\frac{N_{0}}{R_{N_{0}}}\left(  \frac{4}{3}\right)
^{l}=\infty,
\]
from (\ref{ab1}) and (\ref{ab2}) we have a contradiction.

The argument to show that $\lim_{n\rightarrow\infty}\left(  R_{n}%
-R_{n-1}\right)  $ can not be infinity is an immediate consequence of the
previous case.
\end{proof}

Note that a simple adaptation of the proof of the above lemmata provides the
following simple but apparently useful general result:

\begin{proposition}
\label{py}If $\left(  R_{n}\right)  _{n=1}^{\infty}$ is well-behaved and
\[
1\leq\lim_{n\rightarrow\infty}\frac{R_{2n}}{R_{n}}<2,
\]
then%
\[
\lim_{n\rightarrow\infty}\left(  R_{n}-R_{n-1}\right)  =0.
\]

\end{proposition}

In particular, this result reinforces that the information (\ref{i1}) implies
(\ref{i2}). Our main result is a straightforward consequence of the previous lemmata:

\begin{theorem}
[Dichotomy]\label{dtt}If $1\leq R_{n}\leq C_{n}$ for all $n,$ exactly one of
the following assertions is true:

(i) $\left(  R_{n}\right)  _{n=1}^{\infty}$ is subexponential and not well-behaved.

(ii) $\left(  R_{n}\right)  _{n=1}^{\infty}$ is well-behaved with
\[
\lim_{n\rightarrow\infty}\frac{R_{2n}}{R_{n}}\in\lbrack1,\frac{e^{1-\frac
{1}{2}\gamma}}{\sqrt{2}}]
\]
and
\[
\lim_{n\rightarrow\infty}\left(  R_{n}-R_{n-1}\right)  =0.
\]

\end{theorem}

\bigskip

As we have just mentioned (it is a consequence of Proposition \ref{py}), the
information $\lim_{n\rightarrow\infty}\left(  R_{n}-R_{n-1}\right)  =0$ in
(ii) is in fact a consequence of the fact that $\left(  R_{n}\right)
_{n=1}^{\infty}$ is well-behaved with $\lim_{n\rightarrow\infty}\frac{R_{2n}%
}{R_{n}}\in\lbrack1,\frac{e^{1-\frac{1}{2}\gamma}}{\sqrt{2}}].$

In the real case it is known that $\left(  K_{n}\right)  _{n=1}^{\infty}$
satisfies
\[
2^{1-\frac{1}{n}}\leq K_{n}\leq C_{n}%
\]
for all positive integer $n$ (see \cite{diniz2}). It is not difficult to
obtain an example of subexponential and not well-behaved sequence satisfying
the above inequality. For example,%
\[
R_{n}=\left\{
\begin{array}
[c]{c}%
2^{1-\frac{1}{n}}\text{, if }n=2^{k}\text{ for some }k,\\
C_{n}\text{, otherwise.}%
\end{array}
\right.
\]

\section{Final remarks}

It is well-known that the powers $\frac{2n}{n+1}$ in the Bohnenblust--Hille
inequality are sharp; so, it is a common feeling that the optimal constants
from the Bohnenblust--Hille inequality must have an uniform behavior, without
strange fluctuations on their growth. The fact that (i) is fulfilled would,
indeed, be a strongly unexpected result. On the other hand, if (ii) is true
(and we conjecture that this is the case) we would also have a noteworthy
information on the growth of these constants:%
\[
\lim_{n\rightarrow\infty}\left(  K_{n}-K_{n-1}\right)  =0,
\]
and this is also a surprising result in view of the previous known estimates
for the growth of these constants (see \cite{bh, defant, defant2, Mu}).

As we mentioned in the previous section, for the case of real scalars we know that
\begin{equation}
K_{n}\geq2^{1-\frac{1}{n}} \label{est22}%
\end{equation}
and $K_{2}=\sqrt{2}$ (and also that $K_{3}>K_{2})$. On the one hand all known
estimates for the constants in the Bohnenblust--Hille inequality indicate that
we \textquotedblleft probably\textquotedblright\ have $\lim_{n\rightarrow
\infty}K_{n}=\infty;$ but, as a matter of fact, we do not know any proof that
the sequence $\left(  K_{n}\right)  _{n=1}^{\infty}$ tends to infinity. Also,
the sequence $\left(  2^{1-\frac{1}{n}}\right)  _{n=1}^{\infty}$ is obviously
well-behaved and it is not completely impossible that the above estimates
(\ref{est22}) are sharp.

As a final remark mention that a consequence of the Dichotomy Theorem asserts
that if
\begin{equation}
q\in\mathbb{R}-\left[  0,\beta\right]  \label{gg}%
\end{equation}
with%
\[
\beta:=\log_{2}\left(  \frac{e^{1-\frac{1}{2}\gamma}}{\sqrt{2}}\right)
\approx0.526
\]
and $c\in(0,\infty)$, then the sequence $\left(  K_{n}\right)  _{n=1}^{\infty
}$ \textit{can not} be of the form%
\[
K_{n}\sim cn^{q}.
\]
In fact, denoting $B_{n}^{(q)}=cn^{q}$ we have
\begin{align*}
\lim_{n\rightarrow\infty}\frac{K_{2n}}{K_{n}}  &  =\lim_{n\rightarrow\infty
}\left(  \frac{B_{2n}^{(q)}}{B_{n}^{(q)}}.\frac{K_{2n}}{B_{2n}^{(q)}}%
.\frac{B_{n}^{(q)}}{K_{n}}\right) \\
&  =\lim_{n\rightarrow\infty}\left(  \frac{B_{2n}^{(q)}}{B_{n}^{(p)}}\right)
.\lim_{n\rightarrow\infty}\left(  \frac{K_{2n}}{B_{2n}^{(q)}}\right)
.\lim_{n\rightarrow\infty}\left(  \frac{B_{n}^{(q)}}{K_{n}}\right) \\
&  =2^{q}.
\end{align*}
Since $q\in\mathbb{R}-\left[  0,\log_{2}\left(  \frac{e^{1-\frac{1}{2}\gamma}%
}{\sqrt{2}}\right)  \right]  ,$ we have
\[
2^{q}\notin\lbrack1,\frac{e^{1-\frac{1}{2}\gamma}}{\sqrt{2}}]
\]
and it contradicts the Dichotomy Theorem. The case $q<0$ is in fact impossible
since we know that $K_{n}$ belongs to $[1,\infty).$

A similar reasoning shows that if $p(n)$ is any non-constant polynomial, then
\[
K_{n}\nsim p(n).
\]

\end{document}